\documentclass[a4paper]{amsart}

\usepackage[english]{babel}
\usepackage[utf8x]{inputenc}
\usepackage[T1]{fontenc}
\usepackage[a4paper,top=3cm,bottom=2cm,left=2.5cm,right=2.5cm,marginparwidth=2cm]{geometry}

\usepackage{amsmath, latexsym, amssymb, amsthm}
\usepackage{graphicx}
\usepackage{enumerate}
\usepackage[colorinlistoftodos]{todonotes}
\usepackage[colorlinks=true, allcolors=blue]{hyperref}

\newtheorem{thm}{Theorem}[section]

\newtheorem{prop}[thm]{Proposition}

\theoremstyle{plain}

\theoremstyle{definition}
\newtheorem{exam}[thm]{Example}

\usepackage{comment}
\usepackage{tikz}
\usepackage{xcolor}
\usepackage{placeins}
\usepackage{float}

\usepackage{booktabs} %需要加载宏包{booktabs}
\usepackage{amsaddr}

\usepackage{commath}
\usepackage{threeparttable}

\usepackage{subfigure}

\newcommand\myTabularSpace{1.2}

\newenvironment{psmallmatrix}
  {\left(\begin{smallmatrix}}
  {\end{smallmatrix}\right)}

\usepackage{makecell}
\usepackage{siunitx, mhchem}

\begin{document}

\title[SEMIDEFINITE Programming BOUNDS FOR SPHERICAL THREE-DISTANCE SETS]{SEMIDEFINITE programming BOUNDS FOR SPHERICAL THREE-DISTANCE SETS}

%% 關鍵字詞
 %MSC2010

%% 時間跟地址
\author{Feng-Yuan Liu$^1$, Wei-Hsuan Yu$^2$}
\address{$^1$Department of Computer Science, National Tsing Hua University \\$^2$Department of Mathematics, National Central University}
%\date{\today}
\email{smilepa3034@gmail.com, u690604@gmail.com}
\maketitle

%% 內文
\pagenumbering{arabic}

\begin{abstract}
\normalsize
A spherical three-distance set is a finite collection $X$ of unit vectors in $\mathbb{R}^{n}$ such that for each pair of distinct vectors has three inner product values. We use the semidefinite programming method to improve the upper bounds of spherical three-distance sets for several dimensions. We obtain better bounds in $\mathbb{R}^7$, $\mathbb{R}^{20}$, $\mathbb{R}^{21}$, $\mathbb{R}^{23}$, $\mathbb{R}^{24}$ and $\mathbb{R}^{25}$. In particular, we prove that maximum size of spherical three-distance sets is $2300$ in $\mathbb R^{23}$.  
\end{abstract}

\section{Introduction}
\label{chap:intro}

A spherical $s$-distance set is a finite collection $X$ of unit vectors in $\mathbb{R}^{n}$ such that for each pair of distinct vectors has $s$ inner product values. 
We address the problem to find maximum size of spherical $s$-distance sets in $\mathbb R^n$. The study can be traced from Delsarte, Goethals and Seidel \cite{delsarte1977spherical}. They proved
that the cardinality of a spherical $s$-distance set $X$ is bounded
above by
\begin{equation}\label{DGS}
    \abs{X} \leq   \binom {n+s-1}{n-1} + \binom {n+s-2}{n-1}.
\end{equation}
For instance, if $s=2$, $\, \abs{X} \leq \frac{n(n+3)}2$ and if $s=3$, $\, \abs{X} \leq \frac{n(n+1)(n+5)}6$. 
Maximum spherical $s$-distance sets are close related to tight spherical $t$-designs. The notion of spherical designs was studied by Delsarte, Gothas and Seidel \cite{delsarte1977spherical}. The definition is as follows. 

A finite set of sphere $X$ is called a \emph{spherical $t$-design} in $\mathbb{R}^n$ if the following equality holds:
$$
\frac 1 {|\mathbb{S}^{n-1}|}\int_{\mathbb \mathbb{S}^{n-1}} f ds = \frac{1}{|X|} \sum_{y \in X} f(y)
$$
, where $f$ is a polynomial with degree at most $t$. Spherical designs can be regarded as the well distributed sampling points on sphere.
Also, the size of spherical $t$-designs has the lower bounds: 
\begin{equation}
 \abs{X} \geq 
\begin{cases}
\binom {n+s-1}{n-1} + \binom {n+s-2}{n-1}, &\mbox{if $t$ is even and $t=2s$.} \vspace{15pt} \\

2\binom {n+s-1}{n-1}, &\mbox{if $t$ is odd and $t=2s+1$.}
\end{cases}
\end{equation}
If the size of a spherical design attains above inequality, then we call it a \emph{tight} spherical design. Furthermore, tight spherical designs only have few distinct distances among them (a tight spherical $2s$-design must be a spherical $s$-distance set) \cite{delsarte1977spherical}. For instance, a tight spherical 6-design is a spherical three-distance set. A tight spherical $7$-design is an antipodal four-distance set. If we take half of it, it will become a spherical three-distance set. The half of an antipodal set means that we only take one point out of a pair of antipodal points.    
    
For $s=2$, i.e. maximum size of spherical two-distance sets, there are bunch of works being done by Musin \cite{musin2009spherical}, Barg-Yu \cite{barg2014}, Yu \cite{yu2017new} and Glazyrin-Yu \cite{glazyrin2018upper}. Basically, almost every dimension for maximum size of spherical two-distance sets is known. Maximum size of spherical two-distance sets in $\mathbb{R}^n$ is $\frac{n(n+1)}2$ with possible exceptions when $n=(2k+1)^2-3$, where $k \in \mathbb{N}$ \cite{glazyrin2018upper}. However, for $s=3$, it is known very little, only solved for dimension $n=2$, $3$, $4$, $8$ and $22$. Moreover, we solve this problem for $\mathbb{R}^{23}$ in this paper; the answer is 2300 points with the inner product values $\{\pm \frac 1 3, 0\}$ (\textbf{Table \ref{tab:table 3dis}}). The configuration is half of a tight spherical 7-design and which is a sharp code (a four-distance set and a 7-design) and also an universal optimal code discussed in Cohn-Kumar \cite{cohn2007uniqueness}.
\begin{table}[H] 
\centering
\begin{threeparttable}
\begin{tabular}{p{0.5cm}p{1cm}p{7.8cm}p{3.8cm}}
\toprule
n & Size & Structure & Inner product value\\
\midrule
2 & 7 & Heptagon (attaining bound \eqref{DGS}) & \\
3 & 12 & Icosahedron \cite{shinohara2013uniqueness} & (-1, -$\sqrt{5}$/5, $\sqrt{5}$/5) \\
4 & 13 & \textbf{Theorem 3.5} of \cite{szollHosi2018constructions} & \\
8 & 120 & Subset of $E_8$ root system \cite{musin2011bounds} & (-1/2, 0, 1/2)\\
22 & 2025 & Subset of the minimum vectors in the Leech lattice \cite{musin2011bounds} & (-4/11, -1/44, 7/22)\\
23 & 2300 & Half of a tight spherical 7-design [*new result] & (-1/3, 0, 1/3)\\
\bottomrule
\end{tabular}

\caption{Known result of max spherical three-distance sets in $\mathbb{R}^{n}$}
\label{tab:table 3dis}
\end{threeparttable}
\end{table}

The tight spherical 7-design in $\mathbb{R}^{23}$ is also called kissing configuration with 4600 points and inner product values $\{-1, 0, \pm \frac 1 3 \}$. Interestingly, we prove that half of this configuration is the maximum spherical three-distance set in $\mathbb{R}^{23}$. The uniqueness of tight spherical 7-design is discussed in \cite{bannai1981uniqueness,cuypers2005note,cohn2007uniqueness}. The tight spherical 7-design can be constructed by the subset of Leech Lattice. $\{ x \in \text{Leech  lattice} : \langle x \cdot e_1 \rangle = \frac{1}{2}, \, e_1 = (1,0,0,...,0) \}$. Leech lattice is an elegant configuration and has a lot of interesting properties. For instance, it is the solution of sphere packing problem \cite{cohn2017sphere} and kissing number problem \cite{bachoc2008new, odlyzko1979new} in $\mathbb{R}^{24}$. We add one more nice property that the slice of Leech lattice gives arise to a maximum spherical three-distance set in one lower dimension.  

Tight spherical 7-designs exist in very special dimensions; only when the dimension $n$ is three times square of an integer minus 4, i.e. $n = 3k^2-4$, where $k \in \mathbb{N}$ and $k\geq 2$ \cite{bannai2009survey}. For $k=2$, i.e. $n=8$, the tight spherical 7-design in $\mathbb{R}^8$ is the root system of $E_8$. Musin and Nozaki \cite{musin2011bounds} had proven that half of it (120 points) forms the maximum spherical three-distance set in $\mathbb{R}^8$. Our result basically continues this story to $k=3$, i.e. $n=23$ and we prove that half of tight spherical 7-designs in $\mathbb{R}^{23}$ is a maximum spherical three-distance set in $\mathbb{R}^{23}$. You may keep thinking what happen for $k \geq 4$? Unfortunately, the existence is not clear for $k \geq 4$, i.e. there do not yet exist any constructions for tight spherical 7-designs in dimensions $n=44, 71, 104, \cdots$. Some cases have been proven the nonexistence results by Bannai, Munemasa and Venkov \cite{bannai2005nonexistence}. However, there are still infinitely many cases remaining open. We believe that any tight spherical 7-designs will give arise to maximum spherical three-distance sets. However, the prove is still elusive.

%The standard inner product of Euclidean $\mathbb{R}^{n}$ is denoted by $\langle x\cdot y \rangle$. Define the unit sphere space $\mathbb{S}^{n-1} \subset \mathbb{R}^{n}$:
%$$\mathbb{S}^{n-1}:= \{x\in\mathbb{R}^{n}: \langle x, x %\rangle = 1\} $$

We define the maximum size of spherical $3$-distance sets in $\mathbb{S}^{n-1}$ by $A(\mathbb{S}^{n-1})$, where $\mathbb{S}^{n-1}$ is the unit sphere in $\mathbb{R}^n$.
Our work is motivated by Barg and Yu \cite{barg2013new}. They used the semidefinite programming (SDP) method to improve the upper bounds for the size of spherical two-distance sets and obtain exact values 276 for dimension $n=23$ and $\frac {n(n+1)}2$ for $40\leq n \leq 93$ except for $n=46,78$. 
Musin and Nozaki \cite{musin2011bounds} used Delsarte's linear programming method to improve the upper bounds for $A(\mathbb{S}^{n-1})$. They obtained the exact answers for $ A(\mathbb{S}^{7})=120$ and $A(\mathbb{S}^{21})=2025$. In addition, they also improved upper bounds for dimensions from $n=6$ to $50$ (e.g. $A(\mathbb{S}^{6})\leq 91$). Our contribution is to achieve tighter upper bounds by the semidefinite programming. We proved that $A(\mathbb{S}^{6})\leq 84$, $A(\mathbb{S}^{19})\leq 1540$, $A(\mathbb{S}^{20})\leq 1771$, $A(\mathbb{S}^{23})\leq 2600$, $A(\mathbb{S}^{24})\leq 2925$, and obtain the exact answer for $\mathbb{R}^{23}, A(\mathbb{S}^{22})=2300$. Readers may see the summarized results in \textbf{Table \ref{tab:Final Result}}.

This paper is organized by the following sections: 
Section \ref{section previous} discusses the previous methods : absolute harmonic bounds and linear programming bounds.
Section \ref{section SDP} introduces the notion of semidefinite programming method. 
Section \ref{section sampling} integrates Nozaki theorem (\textbf{Theorem \ref{Nozaki thm}}) and SDP/LP methods to do sampling on the $(0,1)$ interval. 
Section \ref{section sos} uses the sum of squares method to make a rigorous proof between the sampling points.
Section \ref{section discussion} includes the discussions and the conclusions. 

\section{Previous methods}\label{chap:methodology1}\label{section previous} 

In this section, we introduce two previous methods to study the upper bounds of spherical three-distance sets: the harmonic absolute bounds and the linear programming bounds. 

Denote the Gegenbauer polynomials of degree $k$ with dimension parameter $n$ by $G_k^{n}(t)$. They are defined with the following recurrence relation:
 $$G^{n}_{0}(t)=1, G^{n}_{1}(t)=t,$$
$$G^{n}_{k}(t)={\frac{(2k+n-4) \, t \, G^{n}_{k-1}(t)-(k-1)G^{n}_{k-2}(t)}{k+n-3}} ,k\geq{2}.$$

For instance, 
\begin{equation*}
G_{2}^n(t) = \frac{n t^{2} - 1}{n - 1}, \quad
G_{3}^n(t) = \frac{t}{n - 1} \left(n t^{2} + 2 t^{2} - 3\right).
\end{equation*}

Then, these Gegenbauer polynomials play the important roles for the harmonic absolute bounds.

\subsection*{Harmonic absolute bound} \label{sec:harmonic bound}
\par Harmonic absolute bound (HB) is proved by Delsarte \cite{delsarte1973algebraic}\cite{delsarte1973four}\cite{levenshtein1992designs}. Later, Nozaki \cite{nozaki2009new} improved this upper bound.

\begin{thm}(Harmonic absolute bound) \cite{nozaki2009new}\label{Thm Harmonic absolute bound}
\par
Let X be a three-distance set in $\mathbb{S}^{n-1} \subset \mathbb{R}^n $ with $D(X)=\{d_1, d_2 ,d_3 \} $, where $D(X)$ collecting all the inner product value of any two distinct points in $X$. Consider the polynomial $f(x)=\prod_{i=1}^3 (d_i-x)$ and suppose that its expansion in the basis ${G^n_k (x)}$ has the form  $f(x)=\sum\limits_{k=0}^3 f^n_k G^n_k (x)$ expanded in the basis ${G^n_k}$, then $$\abs{X} \leq \sum_{k:f^n_k>0} h^n_k,$$ 
where $h^n_k={\binom{n+k-1}{k}}-{\binom{n+k-3}{k-2}}$ which is the dimension of linear space on all real harmonic homogeneous polynomials of degree $k$. %\\
\end{thm}

If the dimension and three inner product values are given, we can calculate the harmonic bounds. See the following example:

\begin{exam}\label{ex:n23_harmonic}
Considering the half of tight spherical 7-designs in $\mathbb{R}^{23}$, the three inner product values in distinct points are $(\frac{-1}{3}, 0, \frac{1}{3})$. We calculate the harmonic bounds with given three inner product values in the following.
\par
Let $f(x) = \prod_{i=1}^3 (d_i-x) = \sum\limits_{k=0}^3 f^n_k G^n_k(x)$, $f^n_k$ can be written as the following term:

\begin{equation*}
\begin{aligned}
f^n_0 &= -d_1 d_2 d_3-{\frac{d_1+d_2+d_3}{n}} \\
f^n_1 &= d_1 d_2+d_1 d_3+d_2 d_3+{\frac{3}{n+2}} \\
f^n_2 &= {\frac{1-n}{n}(d_1+d_2+d_3)} \\
f^n_3 &= {\frac{n-1}{n+2}} \\
\end{aligned}
\end{equation*}

Substitute $f^n_k$ with the dimension $n = 23$, $(d_1, d_2, d_3) = (\frac{-1}{3}, 0, \frac{1}{3})$, then $f^n_1, f^n_3 > 0$, $f^n_0, f^n_2 \leq 0$. Thus, the harmonic absolute bound is $$\abs{X} \leq \sum_{k:f^n_k>0} h^n_k = h^n_1 + h^n_3 = 23 + 2277 = 2300. $$ 
\end{exam}

\subsection*{Linear programming bound} \label{LP bound}\label{sec:LP bound}
\par
Linear programming (LP) is another method to estimate the upper bound of spherical few distance sets. This theorem is established by Delsarte \cite{delsarte1977spherical}. Musin and Nozaki \cite{musin2011bounds} incorporate LP method and Nozaki theorem (\textbf{Theorem \ref{Nozaki thm}}) to obtain the upper bound of spherical codes.

\begin{thm}(Delsarte's inequality) \cite{delsarte1977spherical}
\label{thm Delsarte Inequality}
\par
For any finite set of points $X\subset \mathbb{S}^{n-1}$
$$\sum\limits_{(x,y)\in X^2}G_k^{n}(\langle x, y \rangle)\geq 0,  \forall k \in \mathbb{N}.$$
\end{thm}
Delsarte proved this inequality by the addition formula for spherical harmonics. With this linear inequality, we can derive Delsarte linear programming bound for the spherical three-distance sets. 

\begin{thm}(Delsarte's linear programming bound) \cite{musin2011bounds}\cite{delsarte1977spherical}\label{thm DelsarteLP}\label{thm:DelsarteLP}
\par
Let $X\in \mathbb{S}^{n-1}$ be a finite set and assume that for any $x,y \in X$, $\langle x, y\rangle \in \{d_1,d_2,d_3\}$. Then the cardinality of $X$ is bounded above by the solution of the following linear programming problem:

\begin{subequations}
\begin{align}
& \underset{}{\text{maximize}}
& & 1 + x_1 + x_2 + x_3 \\
& \text{subject to}
& & 1 + x_1 G_k^{n}(d_1) + x_2 G_k^{n}(d_2) + x_3 G_k^{n}(d_3)\geq 0, \forall k \in \mathbb{N}  \label{eq:LP_Delsart_inequality}\\ 
&&& x_j\geq0, j=1,2,3. \label{eq:LP_Nature_conditon} 
\end{align}
\end{subequations}
\end{thm}

Therefore, if the dimension $n$ and the three inner product values $d_1,d_2,d_3$ are given, we can solve above LP problem to obtain the upper bounds for the size of a spherical three-distance set. For instance, if we set $n=23, (d_1,d_2,d_3)=(-1/3,0,1/3)$ (half of a tight $7$-design in $\mathbb{R}^{23}$) and $k \leq 18$, we can obtain the upper bound $2300$ which is coherent to the harmonic bound Example \ref{ex:n23_harmonic}.  

% \begin{thm}
% 這裡之後補上musin, nozaki的定理
% \end{thm}

%%%%%%%%目前改到這邊%%%%%%%%%%%%%%%%%%%%

%%\input{Sec3:SemidefiniteProgramming}
\section{Semidefinite programming method}
\label{chap:methodology2}\label{section SDP}

 %These matrix constraints are derived from Schoenberg theorem. \cite{schoenberg1942positive}

\subsection{Semidefinite Programming}\label{sec:SDP bound}
  A semidefinite program (SDP) is an optimization problem
of the form
 \cite{vandenberghe1996semidefinite}
\begin{equation*}
\begin{aligned}
& \underset{}{\text{minimize}}
& & c^Tx \\
& \text{subject to}
& & F(x)\succeq 0, 
\end{aligned}
\end{equation*}

where 
$$F(x)\triangleq F_0+\sum\limits_{i=1}^m x_iF_i.$$
The vector $c \in \mathbb{R}^m$ and $F_0$, $\cdots$, $F_m$ are symmetric matrices in $\mathbb{R}^{n\times n}$. The inequality sign in $F(x)\succeq0$ means that $F(x)$ is positive semidefinite, i.e.,  
$$z^{T}Fz \geq 0, \forall z \in \mathbb{R}^{n}.$$

SDP is an extension of the linear programming which has been used on bounding the size of codes with a given set of restrictions. For instance, the kissing number problem is asking that how many unit sphere that can touch the center one unit sphere without overlapping. This problem is equivalent to ask the maximum size of spherical codes such that each pair of points has inner product values in the interval $[-1,1/2]$.  %It is one of the convex optimization \cite{boyd2004convex}. 
Bachoc-Vallentin \cite{bachoc2008new} used SDP to improve the upper bounds for the kissing number problem. We adapted their formula to the case of spherical three-distance sets. The SDP method might get tighter upper bounds than linear programming method, since we use the matrix constraints which is additional to the linear constraints.
%\subsection*{Schoenberg theorem with SDP} 
%\begin{thm} \cite{schoenberg1942positive}

Following \cite{bachoc2008new},
define the matrices $Y_k^n(u,v,t)$ and $S_k^n(u,v,t)$ with size $(p_{SDP}-k+1)\times(p_{SDP}-k+1)$, where $p_{SDP}$ is the parameter of SDP matrix constraints.

For all $0 \leq i,j \leq p_{SDP}-k,$ 
$$(Y_k^n(u,v,t))_{ij}=u^iv^j((1-u^2)(1-v^2))^{\frac{k}{2}}G_k^{n-1}(\frac{t-uv}{\sqrt{(1-u^2)(1-v^2)}}), 0 \leq k \leq p_{SDP},$$ 
$$S_k^n(u,v,t) =\frac{1}{6} \sum\limits_{\sigma \in S_3}Y_k^n(\sigma(u,v,t)), $$ 
$$S_k^n(1,1,1) = \mathtt{0}_{M}, k\geq 1,$$

where $\sigma(u,v,t)$ is all the permutation over permutation group and $S_3$. $\mathtt{0}_{M}$ is the zero matrix.

Then
$$\sum\limits_{(x,y,z)\in X^3} S_k^n(x\cdot y,x\cdot z,y\cdot z)\succeq 0.$$

%We can combine the result of Schoenberg theorem (\textbf{Theorem \ref{thm:Schoenberg}}) and Delsarte's linear programming bounds (\textbf{Theorem \ref{thm:DelsarteLP}}) to have a new upper bound - Semidefinite programming bound .

\subsection*{Semidefinite programming bound} \label{SDP bound}
  
Barg and Yu \cite{barg2013new} used SDP to obtain the upper bounds of spherical two-distance sets. We extend their approach to three-distance set to achieve the upper bounds for spherical three-distance sets. 

\begin{thm} \label{thm SDP primal}
Let $p_{LP}, p_{SDP}$ be the parameter of LP constraints and SDP maxtix constraints.\footnote{ In our paper, we set $(p_{LP}, p_{SDP}) = (18, 6)$. Besides, we also tried with $(p_{LP}, p_{SDP}) = (18, 5)$, then our experiment is not able to get the upper bound $A(\mathbb{S}^{22}) \leq 2300$. We conclude that the matrix condition with $p_{SDP} = 6$ is crucial, though we do not have theoretical explain.} If $X$ is a spherical three-distance set with inner product values $d_1,d_2,d_3$, the cardinality of $X$ is bounded above by the solution of the following semidefinite programming problem:
\begin{subequations}
\begin{align}
& \underset{}{\text{maximize}}
& & 1+\frac{1}{3}(x_1+x_2+x_3) \label{cond_SDP:1}\\
& \text{subject to}
& & \begin{pmatrix}
1&0\\
0&0
\end{pmatrix}
+{\frac{1}{3}}
\begin{pmatrix}
0&1\\
1&1
\end{pmatrix}
(x_1+x_2+x_3) +
\begin{pmatrix}
0&0\\
0&1
\end{pmatrix}
\sum_{i=4}^{13} x_i 
\succeq 0, \label{cond_SDP:2}\\
&&& 3+x_1 G_{k}^{n}(d_1)+x_2 G_{k}^{n}(d_2)+x_3 G_{k}^{n}(d_3)\geq 0, \, k=1,2,\cdots,p_{LP}, \label{cond_SDP:3}\\
&&& 
\begin{aligned}
&S_{k}^{n}(1,1,1)+x_1 S_{k}^{n}(d_1,d_1,1)
+x_2 S_{k}^{n}(d_2,d_2,1)+x_3 S_{k}^{n}(d_3,d_3,1)\\
&+x_4 S_{k}^{n}(d_1,d_1,d_1)
+x_5 S_{k}^{n}(d_2,d_2,d_2)
+x_6 S_{k}^{n}(d_3,d_3,d_3)\\
&+x_7 S_{k}^{n}(d_1,d_1,d_2)
+x_8 S_{k}^{n}(d_1,d_1,d_3)
+x_9 S_{k}^{n}(d_2,d_2,d_1)\\
&+x_{10} S_{k}^{n}(d_2,d_2,d_3)
+x_{11} S_{k}^{n}(d_3,d_3,d_1)
+x_{12} S_{k}^{n}(d_3,d_3,d_2)\\
&+x_{13} S_{k}^{n}(d_1,d_2,d_3)  \succeq  0, \, k=0,1,2,\cdots,p_{SDP}, 
\end{aligned} \label{cond_SDP:4} \\ 
&&& x_j  \geq 0, \, j=1,2,\cdots,13. \label{cond_SDP:5}
\end{align}
\end{subequations}
In this theorem, the variables $x_i$ refer to the number of triple points in $X$ associated to some combinations of inner product $d_1,d_2,d_3$ and $1$. For instance, $x_1$ is related to the counting of triple points in $X$ such that the three inner product values are $(d_1,d_1,1)$. Similar setting, $x_2$ is for $(d_2,d_2,1)$ and so on to $x_{13}$ linking to $(d_1,d_2,d_3)$.
\end{thm}

\section{Discrete sampling points with Nozaki theorem}
\label{chap:methodology3}\label{section sampling}
With given three inner product values $d_1, d_2, d_3$, we can consider the harmonic absolute bound (\textbf{Section \ref{sec:harmonic bound}}), the linear programming bound (\textbf{Section \ref{sec:LP bound}}), and the semidefinite programming bound (\textbf{Section \ref{sec:SDP bound}}). Furthermore, when the size of $X$ is big enough, there are some relations among $d_1, d_2, d_3$. In this section, we introduce Nozaki's Theorem. With this result, we can write $d_1, d_2$ as the function of $d_3$, then we can reduce three variables $d_1, d_2, d_3$ to uni-variable $d_3$. After reducing to uni-variable $d_3$, we do sampling points in the interval $d_3 \in (0,1)$ \footnote{If $d_1 < d_2 < d_3$ and $d_3 < 0$, then the cardinality of the set is at most $2n+1$ by Rankin's bound \cite{rankin1947closest}.}.

First, define $$(K_1, K_2,K_3) = (\frac{(d_2-1)(d_3-1)}{(d_2-d_1)(d_3-d_1)},\frac{(d_1-1)(d_3-1)}{(d_1-d_2)(d_3-d_2)},\frac{(d_1-1)(d_2-1)}{(d_1-d_3)(d_2-d_3)}). $$ 
\par Then, if the size of spherical three-distance set is large enough, then $K_i$ will be integers and bounded above.  

\begin{thm}(Nozaki Theorem) \cite{nozaki2011generalization}
\par \label{Nozaki thm} If $X$ is a spherical three-distance set and the size of $X$ is greater than or equal to $2N(\mathbb{S}^{n-1})$, then $(K_1, K_2, K_3)$ are all integers and $K_i$ is bounded by the following: $$\abs{K_i}\leq\lfloor{1/2}+\sqrt{N(\mathbb{S}^{n-1})^2/(2N(\mathbb{S}^{n-1})-2)+1/4}\rfloor, \, i = 1, 2, 3. $$ where $N(\mathbb{S}^{n-1}) := h^n_0+h^n_1+h^n_2$.
\end{thm}

The numbers $(K_1, K_2, K_3)$ also satisfy the following equation \cite{musin2011bounds}. 
\begin{equation}
\label{equation: Nozaki Theorem}
\left\{
\begin{aligned}
K_1 + K_2 + K_3 &= 1 \\
d_1K_1 + d_2K_2 + d_3K_3 &= 1 \\
d_1^{2}K_1 + d_2^{2}K_2 + d_3^{2}K_3 &= 1.
\end{aligned}
\right.
\end{equation}

\par By the observation of $(K_1, K_2, K_3)$ and simple calculation, we can get more properties among $(K_1, K_2, K_3)$. (\textbf{Prop \ref{prop Ki}})

\begin{prop} \label{prop Ki}
$K_i$ has the following properties (when $-1 \leq d_1 < d_2 < d_3 < 1$):
\begin{enumerate}
\item $\abs{K_1} < \, \abs{K_2}$
\item $K_1K_2 < 0$
\item $K_1\neq0, K_2\neq0,  K_3\neq0$.
\end{enumerate}
\end{prop}

Moreover, we can solve the system of equations \ref{equation: Nozaki Theorem}. Without loss of generality, we suppose that $d_1 < d_2 < d_3$ and then we can get the roots by $\mathtt{Matlab}$ $\mathtt{Symbolic}$ $\mathtt{Toolbox}$ \cite{matlabsymbolic}: $$(d_1, d_2) = ({\frac{K_1 - d_3 K_1 K_3-(d_3-1)\sqrt{-K_1 K_2 K_3}}{K_1(K_1 + K_2)}}, {\frac{K_2 - d_3 K_2 K_3+(d_3-1)\sqrt{-K_1 K_2 K_3}}{K_2(K_1 + K_2)}})$$
or $$(d_1^{*}, d_2^{*}) = ({\frac{K_1 - d_3 K_1 K_3+(d_3-1)\sqrt{-K_1 K_2 K_3}}{K_1(K_1 + K_2)}}, {\frac{K_2 - d_3 K_2 K_3-(d_3-1)\sqrt{-K_1 K_2 K_3}}{K_2(K_1 + K_2)}})$$

\par However $(d_1^{*}, d_2^{*})$ is invalid. 
By proposition \ref{prop Ki}, 
$$d_1^{*}-d_2^{*} = \frac{(d_3 - 1)\sqrt{-K_1 K_2 K_3}}{K_1 K_2} \geq 0.$$

It makes a contradiction with initial supposition $d_1^{*}<d_2^{*}$. Therefore, the following theorem holds.

\begin{thm}
Suppose $d_1<d_2<d_3$. By the system of equations \ref{equation: Nozaki Theorem}, $d_1,d_2$ can be solved in the following formula:

\begin{equation*}
\begin{aligned}
d_1 &= {\frac{K_1 - d_3 K_1 K_3-(d_3-1)\sqrt{-K_1 K_2 K_3}}{K_1(K_1 + K_2)}} \\
d_2 &= {\frac{K_2 - d_3 K_2 K_3+(d_3-1)\sqrt{-K_1 K_2 K_3}}{K_2(K_1 + K_2)}}
\end{aligned}
\end{equation*}
\end{thm}

\begin{thm} (Improved bounds of \cite{musin2011bounds} with SDP)
\par
Let $\mathfrak{D}(\mathbb{S}^{n-1})$ be the set of all possible spherical three distances $D(X) = \{d_1,d_2,d_3\}$ such that $(K_1, K_2, K_3)$ are integers. For each $D\in \mathfrak{D}(\mathbb{S}^{n-1})$, we have two bounds: Harmonic absolute bound HB [\ref{Thm Harmonic absolute bound}] and semidefinite programming bound SDP [\ref{SDP bound}]. The following result holds:
\par
Let $$B(D):=\min_{D \in \mathfrak{D}(\mathbb{S}^{n-1})} \{ SDP(D), HB(D)\}, \, \footnote{ In Musin and Nozaki's paper \cite{musin2011bounds}, they use two bounds: harmonic absolute bound $HB$ and Delsarte's linear programming bound $LP$. Therefore, they define the upper bound $B(D) := \min_{D \in \mathfrak{D}(\mathbb{S}^{n-1})} \{LP(D), HB(D)\}. $ }$$ 
\par
then $$A(\mathbb{S}^{n-1})\leq \max_{D \in \mathfrak{D}(\mathbb{S}^{n-1})} \{B(D),2N(\mathbb{S}^{n-1})-1\} . $$
\end{thm}

We show the upper bound in the following figure: 

\begin{figure}[H]
    \center{\includegraphics[width=\textwidth]
    {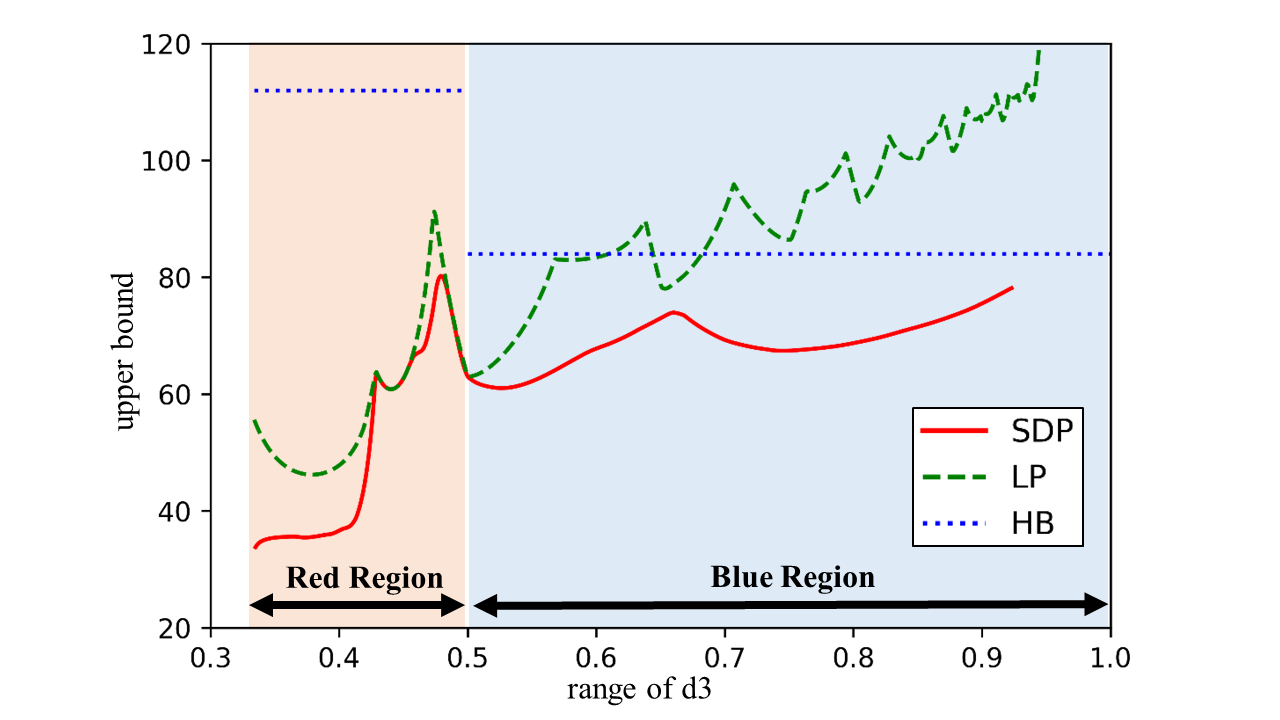}}
    \caption{
    Sampling points figure with $(K_1, K_2, K_3) = (1,-3,3)$ on $\mathbb{R}^{7}$
    \label{fig:N7New} }
\end{figure}

In the blue region of \textbf{Figure \ref{fig:N7New}}, LP/SDP bound will be unbounded as $d_3$ is close to $1$. Thus, we use harmonic bound to control this area. Furthermore, the harmonic bound in this area is always less than or equal to $h_1 + h_3$. On the other hand, the performance of LP/SDP bounds are always better than harmonic bound in the red region. However, in some cases of $(K_1, K_2, K_3)$, LP bound is worse than the harmonic bound of the blue region. In fact, SDP plays a crucial role to dominate the red region with much smaller upper bound. We can see how SDP bounds perform better than LP bounds in the following table [\ref{tab:SDP<LP}]:

%% Table

\begin{table}[H]
\begin{tabular}{rrrrr}
\toprule
 Dimension & $(K_1, K_2, K_3)$ &     Max LP of Red region &    Max SDP of Red region &  HB of Blue region \\
\midrule
         7 &  (1, -3, 3) &    (91.22, ($d_3$ = 0.474)) &    (80.23, ($d_3$ = 0.479)) &                 84 \\
\specialrule{0em}{\myTabularSpace pt}{\myTabularSpace pt}
        20 &  (1, -4, 4) &  (1589.65, ($d_3$ = 0.543)) &   (756.18, ($d_3$ = 0.543)) &               1540 \\
\specialrule{0em}{\myTabularSpace pt}{\myTabularSpace pt}
        21 &  (2, -6, 5) &  (1867.02, ($d_3$ = 0.420)) &  (1332.83, ($d_3$ = 0.420)) &               1771 \\
\specialrule{0em}{\myTabularSpace pt}{\myTabularSpace pt}
        23 &  (1, -3, 3) &   (2385.60, ($d_3$ = 0.590)) &  (1072.29, ($d_3$ = 0.593)) &               2300 \\
        23 &  (2, -6, 5) &  (2319.82, ($d_3$ = 0.421)) &  (1693.01, ($d_3$ = 0.430)) &               2300 \\
        23 &  (3, -8, 6) &  (2300.03, ($d_3$ = 0.332)) &  (2298.12, ($d_3$ = 0.333)) &               2300 \\
\specialrule{0em}{\myTabularSpace pt}{\myTabularSpace pt}
        24 &  (1, -5, 5) &  (2821.84, ($d_3$ = 0.500)) &  (1594.81, ($d_3$ = 0.500)) &               2600 \\
        24 &  (1, -4, 4) &  (2681.29, ($d_3$ = 0.556)) &  (1759.73, ($d_3$ = 0.556)) &               2600 \\
        24 &  (1, -3, 3) &   (2758.20, ($d_3$ = 0.589)) &   (1293.60, ($d_3$ = 0.596)) &               2600 \\
\specialrule{0em}{\myTabularSpace pt}{\myTabularSpace pt}
        25 &  (1, -5, 5) &  (4138.41, ($d_3$ = 0.511)) &  (2472.46, ($d_3$ = 0.522)) &               2925 \\
        25 &  (1, -4, 4) &  (3210.08, ($d_3$ = 0.559)) &  (2238.27, ($d_3$ = 0.559)) &               2925 \\
        25 &  (1, -3, 3) &  (3166.53, ($d_3$ = 0.588)) &  (1883.75, ($d_3$ = 0.600)) &               2925 \\
\bottomrule
\end{tabular}
\caption{Sampling points table with SDP bounds perform better than LP bounds \label{tab:SDP<LP} }
\end{table}

%%%%%%%%%%%%%%
%%%%%%%%%%%%%%

There are many possibilities of $(K_1, K_2, K_3)$ in $\mathbb{R}^n$ when the dimension $n$ is given. Our SDP method only improves the upper bounds of some $(K_1, K_2, K_3)$, but these improvements indeed make a contribution on improving the overall upper bounds. For instance, for $n=23$ and $K_i=(1,-3,3)$ and $(2,-6,5)$, the LP bounds are $2385$ and $2319$ respectively. SDP bounds are $1072$ and $1693$ respectively which are the crucial parts that SDP can prove the upper bounds to $2300$, but LP can not. There are $338$ different possible choices of $(K_1, K_2, K_3)$ in these dimensions. We improve $12$ cases of $(K_1, K_2, K_3)$ of them. (\textbf{Table \ref{tab:Ki Number}})

\begin{table}[H]
\centering
\begin{tabular}{rrr}
\toprule
Dimension & Improved $(K_1, K_2, K_3)$ & Total $(K_1, K_2, K_3)$ \\
\midrule
$7$ & $1$ & $6$\\
$20$ & $1$ & $55$\\
$21$ & $1$ & $55$\\
$23$ & $3$ & $66$\\
$24$ & $3$ & $78$\\
$25$ & $3$ & $78$\\
\bottomrule
\end{tabular}
\caption{Numbers of improved $(K_1, K_2, K_3)$ with respect to total numbers}
\label{tab:Ki Number}
\end{table}

Concatenating the SDP bound of red region and the harmonic bound of blue region, we can get new upper bound on $\mathbb{R}^7$, $\mathbb{R}^{20}$, $\mathbb{R}^{21}$, $\mathbb{R}^{23}$, $\mathbb{R}^{24}$ and $\mathbb{R}^{25}$. (\textbf{Table \ref{tab:Final Result}})

\begin{table}[H]
\centering
\begin{tabular}{rrr}
\toprule
Dimension & Original bound \cite{musin2011bounds} & New bound \\
\midrule
$7$ & $91$ & $84$\\
$20$ & $1541$ & $1540$\\
$21$ & $1772$ & $1771$\\
$23$ & $2301$ & $2300$\\
$24$ & $2601$ & $2600$\\
$25$ & $2926$ & $2925$\\
\bottomrule
\end{tabular}
\caption{Our new result on upper bounds of max spherical three-distance sets}
\label{tab:Final Result}
\end{table}

\section{Rigorous proof with sum of squares (SOS) decomposition} \label{section sos}
We show the SDP bounds on the sampling points in the previous section. However, there might be big oscillation for the upper bounds between the sampling points. We need the rigorous proof to ensure the upper bound between the sampling points also bounded well. We extend the approach from Barg and Yu \cite{barg2013new}. They proved the upper bound of spherical two-distance sets with sum of squares decomposition method. 

In the first step, we write original SDP problem (\textbf{Thm \ref{thm SDP primal}}) into the dual form:

\begin{thm}(SDP Dual Form of Theorem \ref{thm SDP primal}) \label{thm SDP dual}
\begin{subequations}
\begin{align}
& \underset{}{\text{minimize}}
& & 1+\{\sum\limits_{i=1}^{p_{LP}} \alpha_i + \beta_{11}+\langle F_0, S_0^n(1,1,1)\rangle \} \\
& \text{subject to}
& & 
\begin{pmatrix}
\beta_{11}&\beta_{12}\\
\beta_{12}&\beta_{22}
\end{pmatrix}\succeq 0 \\
&&& 2\beta_{12}+\beta_{22}+\sum_{i=1}^{p_{LP}}(\alpha_iG_i^{n}(d_1))+3\sum_{i=0}^{p_{SDP}} \langle F_i , S_i^n(d_1,d_1,1)\rangle\leq -1 \label{SOS_con1} \\
&&& 2\beta_{12}+\beta_{22}+\sum_{i=1}^{p_{LP}}(\alpha_iG_i^{n}(d_2))+3\sum_{i=0}^{p_{SDP}} \langle F_i , S_i^n(d_2,d_2,1)\rangle\leq -1 \label{SOS_con2} \\
&&& 2\beta_{12}+\beta_{22}+\sum_{i=1}^{p_{LP}}(\alpha_iG_i^{n}(d_3))+3\sum_{i=0}^{p_{SDP}} \langle F_i , S_i^n(d_3,d_3,1)\rangle\leq -1 \label{SOS_con3} \\
&&& \beta_{22}+\sum_{i=0}^{p_{SDP}}\langle F_i,S_i^n(y_1,y_2,y_3)\rangle\leq 0 \label{SOS_con4}\\ 
&&& \alpha_i\geq 0, \, i=1,2,\cdots,p_{LP} \\
&&& F_i\succeq 0, \, i=0,1,2,\cdots,p_{SDP}
\end{align}
\end{subequations}

where 
\begin{equation*}
\begin{aligned}
(y_1,y_2,y_3)\in\{
& (d_1,d_1,d_1),(d_2,d_2,d_2),(d_3,d_3,d_3), (d_1,d_1,d_2),(d_1,d_1,d_3), \\
& (d_2,d_2,d_1),(d_2,d_2,d_3),(d_3,d_3,d_1),(d_3,d_3,d_2),(d_1,d_2,d_3)\}
\end{aligned}
\end{equation*}
\end{thm}

Constrains (\ref{SOS_con1})(\ref{SOS_con2})(\ref{SOS_con3})(\ref{SOS_con4}) impose positive conditions on the uni-variate polynomials of $d_3$ ($d_1, d_2$ can be written as the functions of $d_3$), denoting $d_3$ as $a$. The following steps can transform the constraints to matrices positive semidefinite conditions. First, we extend the non-negative polynomial valid from the original small interval $a \in [a_1, a_2]$ to whole real numbers by the following theorem:

\begin{thm} 
If $f(a)$ is a polynomial of degree $m$ satisfies $f(a) \geq 0$ for $a \in [a_1,a_2]$, then
$$f^+(a)=(1+a^2)^mf(\frac{a_1+a_2a^2}{1+a^2})\geq 0, \, \forall a\in \mathbb{R}$$ 
\end{thm}

\begin{proof}
Define $g(a) = \frac{a_1 + a_2 a^2}{1 + a^2} $ and consider the values of g on $a \in [0,\infty)$. The function values on the boundary points of the interval $[0,\infty)$ are $g(0) = a_1$, and $\lim_{a \to \infty} g(a) = a_2$. $g$ is an increasing function for $a \in [0,\infty)$ since $g'(a) = \frac{2a(a_2-a_1)}{(a^2+1)^2} \geq 0$ . Consequently, $f(g(a))$ is non-negative for $a \in [0,\infty)$. Besides, $g(a)$ is an even function, $f(g(a))$ is also non-negative for $a \in (-\infty, 0]$. Therefore, $f^+(a)=(1+a^2)^mf(g(a))$ is non-negative for all real numbers.
\end{proof}

Second, Hilbert proved that a non-negative polynomial on whole real numbers can be written as the sum of squares \cite{hilbert1888darstellung}.
$f^{+}(a)=\sum_ir_i^2(a)$, where $r_i$ are polynomials. Finally, according to the result of Nesterov \cite{nesterov2000squared}, a polynomial $f^{+}(a)$ can be written as sum of squares form if and only if there exists a positive semidefinite matrix $Q$ such that $f^{+}=XQX^T$, where $X=(1,a,a^2,\cdots,a^m)$. Therefore, constraints (\ref{SOS_con1})(\ref{SOS_con2})(\ref{SOS_con3})(\ref{SOS_con4}) can be transformed into positive semidefinite matrices conditions. 

Finally, if there are some positive semidefinite matrices $M_1, M_2, ..., M_k$, then we can deduce that the block matrix $ \begin{psmallmatrix} 
M_1 &       &       &  0   \\
    & M_2   &       &     \\
    &       & ... &     \\
0    &       &       & M_k  \\
\end{psmallmatrix}$ is also positive semidefinite. We use this technique to make our implementation of \textbf{sos decomposition} more convenient with $\mathtt{YALMIP}$ \cite{Lofberg2009}, and show the result in table \ref{tab:Sos Decompostion}. The sos decomposition technique can guarantee that there will not be problems between the sampling points. The values are all smaller than the harmonic bound of blue region, then we complete the rigorous proof. 

\begin{table}[H]
\centering
\begin{tabular}{rllr}
\toprule
 Dimension & $(K_1, K_2, K_3)$ & \thead{Sos Decompostion Value \\(covered Max SDP in Red Region)} &  \thead{Harmonic Bound \\ (Blue Region)} \\
\midrule
         7 &  (1, -3, 3) &                          sos(0.475, 0.480) = 80.29 &                 84 \\
        20 &  (1, -4, 4) &                         sos(0.540, 0.545) = 804.06 &               1540 \\
        21 &  (2, -6, 5) &                        sos(0.415, 0.420) = 1343.66 &               1771 \\
        23 &  (1, -3, 3) &                        sos(0.590, 0.595) = 1234.62 &               2300 \\
        23 &  (2, -6, 5) &                        sos(0.425, 0.430) = 1703.71 &               2300 \\
        23 &  (3, -8, 6) &                        sos(0.332, 0.335) = 2300.85 &               2300 \\
        24 &  (1, -5, 5) &                        sos(0.495, 0.500) = 1594.80 &               2600 \\
        24 &  (1, -4, 4) &                        sos(0.555, 0.560) = 2159.78 &               2600 \\
        24 &  (1, -3, 3) &                        sos(0.595, 0.600) = 1605.49 &               2600 \\
        25 &  (1, -5, 5) &                        sos(0.520, 0.525) = 2495.09 &               2925 \\
        25 &  (1, -4, 4) &                        sos(0.555, 0.560) = 2474.14 &               2925 \\
        25 &  (1, -3, 3) &                        sos(0.595, 0.600) = 2080.54 &               2925 \\
\bottomrule
\end{tabular}
\caption{Sos Decomposition Value Covered Max SDP}
\label{tab:Sos Decompostion}
\end{table}

% Combined with Musin and Nozaki's linear programming \cite{musin2011bounds} and our semidefinite programming method, we can generalize the numerical optimization outcome on spherical three distance sets with the following theorem:
% \begin{thm}
% \begin{enumerate} 
% \item $A(S^{7}, 3) = 120, A(S^{21}, 3) = 2025, A(S^{22}, 3) = 2300$.
% \item $A(S^{3}, 3) \leq 27, A(S^{4}, 3) \leq 39, A(S^{6}, 3) \leq 84$.
% \item For $n = 6$ or $9 \leq n \leq 21$ or $24 \leq n \leq 25$, we have $A(S^{n-1}, 3) \leq h_1 + h_3 = n(n+1)(n+2)/6$.
% \item For $n = 22$ or $26 \leq n \leq 30$, we have $A(S^{n-1}, 3) \leq h_0 + h_1 + h_3 = (n+3)(n^2+2)/6$.
% \item For $31 \leq n \leq 50$, we have $A(S^{n-1}, 3) \leq h_2 + h_3 = (n^2-1)(n+6)/6$.
% \end{enumerate}

%% \input{Sec7:Discussion}
\section{Discussions and conclusions}
\label{chap:conclusion}
\label{section discussion}

We obtain that the SDP bound for spherical three-distance set in $\mathbb{R}^{23}$ is $2300$ and we also know the construction of such $2300$ points from tight spherical 7-designs. Therefore, we are curious whether the construction of spherical three-distance sets with 84 points in $\mathbb{R}^{7}$ exist or not. At first, we consider the orbit of a group. Eiichi Bannai suggested us to try the group PSL$(2,7)$ which has order $168$ just double of our target, $84$ points. We may find the nice initial vector such that the orbit of the group action by PSL$(2,7)$ forms a three-distance set having 84 points. However, the orbit of the irreducible or reducible representation of PSL$(2,7)$ in 7 dimension can not produce such three-distance sets. Second, we do the numerical non-linear optimization to calculate the energy minimizing configuration of $84$ points on $\mathbb{S}^{6}$.
Unfortunately, these two methods don't help us to find the construction. We will try harder to find a powerful algorithm to find this construction since we believe such construction exists for the 84 points to be a spherical three-distance set in $\mathbb{R}^{7}$. If we can find such construction, then in conjunction of our new bounds, we can prove that $A(\mathbb{S}^{6})=84$. Currently, the maximum size of known construction for spherical three-distance sets on $\mathbb{S}^{6}$ is $64$ which can be obtained from $3$-class association scheme with
 64 vertices, and Krein array [7, 6, 5; 1, 2, 3]. The half of $E_7$ root system is also a spherical three-distance in $\mathbb{R}^7$ with $63$ vertices and apparently this can not be the maximum.     

We also try to estimate the upper bounds without using the integer conditions of $d_i$ (Nozaki Theorem \ref{Nozaki thm}). In $\mathbb{R}^4$ and $\mathbb{R}^5$, we run SDP on sampling points of all possible inner product values, i.e. cut the intervals of $d_1, d_2, d_3$ into many parts. We cut $d_1 \in (-1,1), d_2 \in (-1,1)$ into 100 parts, and cut $d_3 \in (0,1)$ into 50 parts. In $\mathbb{R}^4$, the numerical SDP bound can be reduced to 26 and LP only can get the upper bound 27 \cite{musin2011bounds}. For example, if we estimate the bounds for $(d_1,d_2,d_3) = (-0.76, -0.16, 0.54)$, LP bound is 27 and SDP bound is 13. For some $d_i$, SDP bound is large and that moment HB is 26. Therefore, we only get the best bound 26 in $\mathbb{R}^4$ when we use SDP method.  However,  \cite{szollHosi2018constructions} proved the bound to be $13$ and had the construction. In $\mathbb{R}^5$, when we do not use Nozaki Theorem \ref{Nozaki thm}, the SDP bounds on some $d_i$ are larger than harmonic bound (45 points). Therefore, we do not improve the result of Musin and Nozaki \cite{musin2011bounds}. For other dimensions, the sampling points results do not improve the upper bounds and we save our time to do the sum of squares experiments.

\section*{Acknowledgments}
\indent
We greatly appreciate Che-Rong Lee generally supporting first author Feng-Yuan during his study in National Tsing Hua University. We thank Eiichi Bannai for the useful discussion. The second author W.-H. Yu is partially supported by the Ministry of Science and Technology of Taiwan (No. 107-2115-M-008-010-MY2).

\bibliographystyle{amsalpha} 
\bibliography{reference}
\end{document}